\documentclass[journal,10pt,onecolumn,draftclsnofoot,twoside]{IEEEtran}
\usepackage{amsmath,amssymb,euscript ,yfonts,psfrag,latexsym,dsfont,graphicx,bbm,color,amstext,wasysym,subfig,flushend,parskip}

\graphicspath{{./},{./figures/},{../figures/}}

\begin{document}
\newtheorem{thm}{Theorem}
\newtheorem{cor}[thm]{Corollary}
\newtheorem{conj}[thm]{Conjecture}
\newtheorem{lemma}{Lemma}
\newtheorem{prop}[thm]{Proposition}
\newtheorem{problem}{Problem}
\newtheorem{proof}[thm]{Proof} 
\newtheorem{remark}[thm]{Remark}
\newtheorem{defn}[thm]{Definition}
\newtheorem{ex}{Example}

\IEEEoverridecommandlockouts
\newcommand{\cE}{{\mathcal E}}
\newcommand{\cF}{{\mathcal F}}
\newcommand{\cX}{{\Phi_1}}
\newcommand{\cY}{{\Phi_0}}
\newcommand{\bI}{{\mathbb I}}
\newcommand{\cG}{{\mathcal G}}
\newcommand{\mC}{{\mathbb C}}
\newcommand{\mE}{{\mathbb E}}
\newcommand{\cD}{{\mathcal D}}
\newcommand{\cL}{{\mathcal L}}
\newcommand{\cP}{{\mathcal P}}
\newcommand{\cS}{{\mathcal S}}
\newcommand{\cSR}{{\mathcal SR}}
\newcommand{\cK}{{\mathcal K}}
\newcommand{\cM}{{\mathcal M}}
\newcommand{\cC}{{\mathcal C}}

\newcommand{\R}{{\mathbb R}}
\newcommand{\mP}{{\mathbb P}}
\newcommand{\C}{{\mathbb C}}
\newcommand{\D}{{\mathbb D}}
\newcommand{\E}{{\mathbb E}}
\newcommand{\mcN}{{\mathcal N}}
\newcommand{\mcR}{{\mathcal R}}
\newcommand{\fD}{{\mathfrak D}}
\newcommand{\fH}{{\mathfrak H}}
\newcommand{\fM}{{\mathfrak M}}
\newcommand{\fP}{{{\mathfrak P}}}
\newcommand{\tr}{\operatorname{trace}}
\newcommand{\diag}{\operatorname{diag}}
\newcommand{\support}{{\it Supp}}
\newcommand{\ignore}[1]{}

\newcommand{\bike}{\color{blue}}
\newcommand{\mike}{\color{magenta}}
\newcommand{\rike}{\color{red}}

\title{\LARGE \bf Relaxed Schr\"odinger bridges and robust network routing 
*}

\author{Yongxin Chen$^{1}$, Tryphon Georgiou$^{2}$, Michele Pavon$^{3}$ and Allen Tannenbaum$^{4}$%
\thanks{*This project was supported by AFOSR grants (FA9550-15-1-0045 and FA9550-17-1-0435), ARO grant (W911NF-17-1-049), grants from the National Center for Research Resources (P41-RR-013218) and the National Institute of Biomedical Imaging and Bioengineering (P41-EB-015902), NCI grant (1U24CA18092401A1), NIA grant (R01 AG053991), a grant from the Breast Cancer Research Foundation and by the University of Padova Research Project CPDA 140897.}%
\thanks{$^{1}$Yongxin Chen is with the Department of Electrical and Computer Engineering, Iowa State University, IA 50011,
Ames, Iowa, USA
{\tt\small yongchen@iastate.edu}}%
\thanks{$^{2}$ Tryphon Georgiou is with the Department of Mechanical and Aerospace Engineering,
University of California, Irvine, CA 92697, USA
{\tt\small tryphon@uci.edu}}%
\thanks{$^{3}$ Michele Pavon is with the Dipartimento di Matematica ``Tullio Levi-Civita",
Universit\`a di Padova, 35121 Padova, Italy
{\tt\small pavon@math.unipd.it}}%
\thanks{$^{4}$ Allen Tannenbaum is with the Departments of Computer Science and Applied Mathematics \& Statistics, Stony Brook University, NY 11794, USA {\tt\small allen.tannenbaum@stonybrook.edu}}
}

\maketitle
\thispagestyle{empty}
\pagestyle{empty}

\begin{abstract} We consider network routing under random link failures with a desired final distribution. We provide a mathematical formulation of a relaxed transport problem where the final distribution only needs to be close to the desired one. The problem is a maximum entropy problem for path distributions with an extra terminal cost. We show that the unique solution may be obtained solving a  {\em generalized Schr\"{o}dinger system}. An iterative algorithm to compute the solution is provided. It contracts the Hilbert metric with contraction ratio less than $1/2$ leading to extremely fast convergence.
\end{abstract}

\maketitle
\thispagestyle{empty}
\pagestyle{empty}
\section{Introduction}
Containing the 2017-18 Southern California wild fires has been a major challenge for CAL FIRE involving dispatching hundreds of fire engines and thousands of fire fighters including some provided by ten other states. Efficiently dispatching the fire engines over a long period of time (the Thomas fire, for instance, burned for more than one month) is a difficult task. The problem can be roughly described as follows: At the initial time $t=0$ we have a certain distribution of fire engines in certain locations (nodes). Within at most $N$ time units, so as to provide the crew shift, the engines must reach through the available road network the various fire locations (other nodes). The distribution must guarantee the minimum force necessary to fight each specific fire. Considering the difficulties and hazards involved in reaching their destination, it seems  reasonable to require that the final distribution of the fire engines be {\em close} (rather than equal) to a desired one. Another spec of the routing plan is robustness  with respect to link failures. This could be accomplished by dispatching engines on different routes even when they have the same target.

In this paper, building on our previous work \cite{CGPT1,CGPT2}, which deal with the case of a fixed terminal distribution, we provide a precise mathematical formulation of the above relaxed problem. It is a maximum entropy problem for probability distributions on the feasible paths with a terminal cost. We study a relaxed version of the usual Schr\"odinger bridge problem without a hard constraint on the terminal marginal but with an extra terminal cost. The solution is obtained by solving iteratively a {\em generalized Schr\"odinger system}. Convergence of the algorithm in the natural projective metric is established. In \cite{HW}, which is a sort of relaxation of \cite{CGP1},  the problem of optimally steering a linear stochastic system with a Wasserstein distance terminal cost was studied. In \cite{CPSV} (see also \cite{JohRin17}), a  regularized transport problem with very general boundary costs is considered and solved through iterative {\em Schr\"odinger-Fortet-Demin-Stephan-Sinkhorn-like} algorithms \cite{S1,S2,For,DS1940,Sin64}. Although our dynamic problem can be reduced to a static one of the form considered in \cite{CPSV} (see Section~\ref{relaxedbridges}), employing a general prior measure {\em on the trajectories} has some advantages.  Indeed, the static formulation solution does not yield immediate by-product information on the new transition probabilities and on what paths the optimal mass flow occurs and is therefore less suited for many network routing applications. Moreover, we want to allow for general prior measures not necessarily of the Boltzmann's type considered in the previous work. Finally, we prove convergence of the iterative algorithm in the Hilbert rather than Thompson metric as it usually provides the best contraction ratio \cite[Theorem 3.4]{bushell1973hilbert}, \cite{KP}.

We model the network through a directed graph and seek to design the routing policy so that the distribution of the commodity at some prescribed time horizon is close to a desired one. The optimal feedback control suitably modifies a prior transition mechanism. We also attempt to implicitly obtain other desirable properties of the optimal policy by suitably choosing a prior measure in a maximum entropy problem for distributions on paths.    Robustness with respect to network failures, namely spreading of the mass as much as the topology of the graph and the final distribution allow, is accomplished by employing as prior transition the {\em adjacency matrix} of the graph. Our intuitive notion of robustness of the routing policy should not be confused with other notions of robustness concerning networks which have been put forward and studied, see e.g. \cite{Albertetal2000,bara2014robustness,olver2010robust,arnoldetal1994,demetrius2005,savla2014robust}.
In particular, in \cite {arnoldetal1994,demetrius2005}, robustness has been defined through a fluctuation-dissipation relation involving the entropy rate.
This latter notion captures relaxation of a process back to equilibrium after a perturbation and has been used to study both financial and biological networks \cite{Sandhu, Sandhu1}. This paper is addressed to  transportation and data networks problems and does not concern equilibrium or near equilibrium cases.

The outline of the paper is as follows. In the next section, we define the relaxed transport problem. In Section~\ref{main}, we state and prove the main result reducing the problem to solving a generalized Schr\"{o}dinger system. In Section~\ref{Hilbert}, we review some fundamental concepts and results concerning Hilbert's projective metric. In Section~\ref{genSch}, we establish existence and uniqueness for the generalized Schr\"{o}dinger system through a contraction mapping principle. Finally, in Section~\ref{algo}, we outline an iterative algorithm to compute the solution and some extensions of the results.

\section{Relaxed Schr\"odinger bridges}\label{relaxedbridges}


Consider a directed, strongly connected aperiodic graph ${\bf G}=(\mathcal X,\mathcal E)$ with vertex set $\mathcal X=\{1,2,\ldots,n\}$ and edge set $\mathcal E\subseteq \mathcal X\times\mathcal X$.  We let  time vary in ${\mathcal T}=\{0,1,\ldots,N\}$, and let ${\mathcal {FP}}_0^N\subseteq\mathcal X^{N+1}$ denote the family of length $N$, feasible paths $x=(x_0,\ldots,x_N)$, namely paths such that $(x_i,x_{i+1})\in\mathcal E$ for $i=0,1,\ldots,N-1$.

We seek a probability distribution $\fP$ on ${\mathcal {FP}}_0^N$ with prescribed initial probability distribution $\nu_0(\cdot)$ and terminal distribution close to $\nu_N(\cdot)$,  such that the resulting random evolution
is closest to a ``prior'' measure $\fM$ on ${\mathcal {FP}}_0^N$ in a suitable sense.
The prior law $\fM$ is induced by the Markovian evolution
 \begin{equation}\label{FP}
\mu_{t+1}(x_{t+1})=\sum_{x_t\in\mathcal X} \mu_t(x_t) m_{x_{t}x_{t+1}}(t)
\end{equation}
with nonnegative distributions $\mu_t(\cdot)$ over $\mathcal X$, $t\in{\mathcal T}$, and weights $m_{ij}(t)\geq 0$ for all indices $i,j\in{\mathcal X}$ and all times. Moreover, to respect the topology of the graph, $m_{ij}(t)=0$ for all $t$ whenever $(i,j)\not\in\mathcal E$.  Often, but not always, the matrix
\begin{equation}\label{eq:matrixM}
M(t)=\left[ m_{ij}(t)\right]_{i,j=1}^n
\end{equation}
does not depend on $t$.
The rows of the transition matrix $M(t)$ do not necessarily sum up to one, so that the ``total transported mass'' is not necessarily preserved.  This occurs, for instance, when $M(t)$ simply encodes the topological structure of the network with $m_{ij}(t)$ being zero or one, depending on whether a certain link exists at each time $t$. It is also possible to take into account the length of the paths leading to solutions which compromise between speading the mass and transporting on shorter paths, see \cite{CGPT1,CGPT2}.
The evolution \eqref{FP} together with the measure $\mu_0(\cdot)$, which we assume positive on $\mathcal X$, i.e.,
\begin{equation}\label{eq:mupositive}
\mu_0(x)>0\mbox{ for all }x\in\mathcal X,
\end{equation}
 induces
a measure $\fM$ on ${\mathcal {FP}}_0^N$ as follows. It assigns to a path  $x=(x_0,x_1,\ldots,x_N)\in{\mathcal {FP}}_0^N$ the value
\begin{equation}\label{prior}\fM(x_0,x_1,\ldots,x_N)=\mu_0(x_0)m_{x_0x_1}(0)\cdots m_{x_{N-1}x_N}(N-1),
\end{equation}
and gives rise to a flow
of {\em one-time marginals}
\[\mu_t(x_t) = \sum_{x_{\ell\neq t}}\fM(x_0,x_1,\ldots,x_N), \quad t\in\mathcal T.\]

We seek a distribution  which is closest to the prior $\fM$ in {\em relative entropy} where, for $P$ and $Q$ measures on $\mathcal X^{N+1}$,  the relative entropy (divergence, Kullback-Leibler index) $\D(P\|Q)$ is
\begin{equation*}
\D(P\|Q)\hspace*{-2pt}:=\hspace*{-2pt}\left\{\begin{array}{ll} \hspace*{-5pt}\sum_{x\in\mathcal X^{N+1}}P(x)\log\frac{P(x)}{Q(x)}, & \hspace*{-5pt}\support (P)\subseteq \support (Q),\\
\hspace*{-5pt}+\infty , & \hspace*{-5pt}\support (P)\not\subseteq \support (Q),\end{array}\right.
\end{equation*}
Here, by definition,  $0\cdot\log 0=0$.
Naturally, while the value of $\D(P\|Q)$ may turn out negative due to miss-match of scaling (in case $Q=\fM$ is not a probability measure), the relative entropy is always jointly convex. Moreover,
\[\D(P\|Q)- \sum_{x\in\mathcal X^{N+1}}P(x)+ \sum_{x\in\mathcal X^{N+1}}Q(x)\ge 0.
\]
Since for probability distributions we have
\[\sum_{x\in\mathcal X^{N+1}}P(x)=1,
\]
minimizing the nonnegative quantity $\D(P\|Q)- \sum_{x}P(x)+ \sum_{x}Q(x)$ over a family of probability distributions $P$, even when the prior $Q$ has a different total mass, is equivalent to minimizing over the same set $\D(P\|Q)$. We are now ready to formalize the problem. Let $\nu_0$ and $\nu_N$ be two probability distributions on $\mathcal X$ and let $\mathcal P(\nu_0)$ be the family of all Markovian probability distributions on $\mathcal X^{N+1}$ of the form (\ref{prior}) with initial marginal $\nu_0$. Rather than imposing the desired final marginal $\nu_N$ as in the standard Schr\"{o}dinger bridge problem, we consider the following ``relaxed problem":
\begin{problem}\label{relbridge}
\begin{subequations}\label{eq:general}
 \begin{eqnarray}
{\rm minimize}\;J(P):=\D(P\|\fM)+\D(p_N\|\nu_N)\\ {\rm over} \; \{P\in {\mathcal P}(\nu_0)\}.
\end{eqnarray}
\end{subequations}
\end{problem}Clearly, we can restrict the  minimization to distributions in ${\mathcal P}_S(\nu_0)$, namely distributions in
${\mathcal P}(\nu_0)$ such that
\begin{equation}\label{support}\support (p_N)\subseteq \support (\nu_N) .
\end{equation}
The connection between the dynamic Problem \ref{relbridge} and a static problem such as those considered in \cite{CPSV}, can be obtained as follows. Let $P$ and $Q$ be two probability distributions on $\mathcal X^{N+1}$. For $x=(x_0,x_1,\ldots,x_N)\in\mathcal X^{N+1}$, consider the multiplicative decomposition
$$P(x)=P_{x_0,x_N}(x)p_{0N}(x_0,x_N),
$$
where
\[P_{\bar{x}_0,\bar{x}_N}(x)=P(x|x_0=\bar{x}_0,x_n=\bar{x}_N)\]
and we have assumed that $p_{0N}$ is everywhere positive on $\mathcal X\times \mathcal X$, and a similar one for $Q$.
We get
\begin{eqnarray}\nonumber
\D(P\|Q)=\sum_{x_0x_N}p_{0N}(x_0,x_N)\log \frac{p_{0N}(x_0,x_N)}{q_{0N}(x_0,x_N)}\\+\sum_{x\in{\cal X}^{N+1}}P_{x_0,x_N}(x)\log \frac{P_{x_0,x_N}(x)}{Q_{x_0,x_N}(x)} p_{0N}(x_0,x_N).\nonumber
\end{eqnarray}
This is the sum of two nonnegative quantities. The second becomes zero if and only if $P_{x_0,x_N}(x)=Q_{x_0,x_N}(x)$ for all $x\in{\cal X}^{N+1}$.  Thus, $P^*_{x_0,x_N}(x)=Q_{x_0,x_N}(x)$. 
Thus, Problem  \ref{relbridge} can be reduced to
 \begin{eqnarray}
{\rm minimize}\;J(P):=\D(p_{0N}\|m_{0N})+\D(p_N\|\nu_N)\\ {\rm over} \; \{p_{0N}:  \sum_{x_N}p_{0N}(\cdot,x_N)=\nu_0(\cdot)\}.
\end{eqnarray}
This argument  extends to the situation where the prior measure mass is not one. We prefer to discuss the original formulation (\ref{eq:general}) for the reasons described in the Introduction.

\section{Main result}\label{main}
We have the following characterization of the solution.
\begin{thm}\label{mainthm}Assume that the matrix
\begin{equation}\label{transition0N}G:=M(N-1) M(N-2)\cdots M(1) M(0)=\left(g_{ij}\right) \end{equation}
has all positive elements  $g_{ij}$. Suppose there exist two functions $\varphi$ and $\hat{\varphi}$ mapping $\{0,1,\ldots,N\}\times\mathcal X$ into the nonnegative reals and satisfying the {\em generalized Schr\"{o}dinger system}
\begin{subequations}\label{eq:generalizedbridge}
\begin{align}\label{eq:generalizedbridgeA}
&\varphi(t,i)=\sum_{j}m_{ij}(t)\varphi(t+1,j), \; 0\le t\le N-1,\\
&\hat\varphi(t+1,j)=\sum_{i}m_{ij}(t)\hat\varphi(t,i),\; 0\le t\le N-1,\label{eq:generalizedbridgeB}\\
\label{eq:generalizedbridgeC}
&\varphi(0,i)\hat\varphi(0,i)=\nu_0(i),\\\label{eq:generalizedbridgeD}
&\varphi(N,j)^2\hat\varphi(N,j)=\nu_N(j).
\end{align}
\end{subequations}
For $0\le t\le N-1$ and $(i,j)\in\mathcal X\times\mathcal X$, we define
\begin{equation}\label{Opttransition1}\pi^*_{ij}(t):=m_{ij}(t)\frac{\varphi(t+1,j)}{\varphi(t,i)}.
\end{equation}
which constitute a family of {\em bona fide} transition probabilities. Then, the solution $\fP^*$ to Problem \ref{relbridge} is unique and given by the Markovian distribution
\begin{equation}\label{optmeasure}
\fP^*(x_0,\ldots,x_N)=\nu_0(x_0)\pi^*_{x_0x_{1}}(0)\cdots \pi^*_{x_{N-1}x_{N}}(N-1).
\end{equation}
\end{thm}
\begin{proof}
Let $\varphi(\cdot,\cdot)$ be {\em space-time harmonic} for the prior transition mechanism, namely satisfy on $0\le t\le N-1$ recursion (\ref{eq:generalizedbridgeA}).
Observe that since $G$ has all positive elements, $\hat{\varphi}(N,i)$ and $\varphi(0,i)$ are positive for all $i\in\mathcal X$. In particular, it then follows from (\ref{eq:generalizedbridgeD}) that $\varphi(N,i)=0$ if and only if $\nu_N(i)=0$. Minimizing $J(P)$ over ${\mathcal P}_S(\nu_0)$ is then equivalent to minimizing over ${\mathcal P}(\nu_0)$ the new index
\begin{align}\nonumber &J'(P):=J(P) + \sum_{x_0}\log\varphi (0,x_0)\nu_0(x_0)\\&\hspace{-.3cm}-\sum_{x_N}\log\varphi(N,x_N)p_N(x_N)+\sum_{x_N}\log\varphi(N,x_N)p_N(x_N),\label{modifiedindex}
\end{align}
where $p_N$ denotes the marginal of $P$ at time $N$ and we have used the convention $0\cdot\log 0=0$. Let $\pi_{ij}(t)$ be the transition probabilities of the measure $P\in {\mathcal P}_S(\nu_0)$. Then, using the multiplicative decomposition (\ref{prior}) for both measures we get the representation
\begin{align}\nonumber
\D({P}\|\fM)&=\D(\nu_0\|\mu_0)\\&+\sum_{k=0}^{N-1}\sum_{x_k}\D(\pi_{x_kx_{k+1}}(k)\|m_{x_kx_{k+1}}(k))p_k(x_k).\label{relentrdec2}
\end{align}
Since $\D(\nu_0\|\mu_0)$ is constant over ${\mathcal P}_S(\nu_0)$ and by the same calculation as in \cite[pp. 7-8]{PT}, we now get that Problem \ref{relbridge} is equivalent to minimizing over ${\mathcal P}_S(\nu_0)$
\begin{eqnarray}\nonumber \hspace*{-5pt}J''(P)&:=&\sum_{k=0}^{N-1}\sum_{x_k}\D\left(\pi_{x_{k}x_{k+1}}(k)\|m_{x_{k}x_{k+1}}(k)\phantom{\frac{\varphi}{\varphi}}\right.\\\nonumber
&&\left. \times \frac{\varphi(k+1,x_{k+1})}{\varphi(k,x_k)}\right)p_{k}(x_{k}) \\&&+\sum_{x_N}\log\left[\frac{p_N(x_N)\varphi(N,x_N)}{\nu_N(x_N)}\right]p_N(x_N).\label{secondmodifiedindex}
\end{eqnarray}
We next prove that
\begin{equation}\label{opttransition1}\pi^*_{ij}(t):=m_{ij}(t)\frac{\varphi(t+1,j)}{\varphi(t,i)}
\end{equation}
constitute a family of transition probabilities. Indeed, $\pi^*_{ij}(t)\ge 0$ and, by (\ref{eq:generalizedbridgeA}),
$$\sum_j \pi^*_{ij}(t)=\sum_jm_{ij}(t)\frac{\varphi(t+1,j)}{\varphi(t,i)}=\frac{\varphi(t,i)}{\varphi(t,i)}=1.
$$
Thus, the first term in the right-hand side of (\ref{secondmodifiedindex}) is nonnegative and we can make it equal to zero  by choosing as new transition probabilities precisely (\ref{opttransition1}). Consider now the probabilities $p^*(t,\cdot)$ defined by the recursion
\begin{equation}\label{FPopt}p^*(t+1,j)=\sum_i\pi^*_{ij}(t)p^*(t,i),\quad p^*(0,i)=\nu_0(i).
\end{equation}
Observe that the second term in the right-hand side of (\ref{secondmodifiedindex}) becomes zero if
\begin{equation}\label{bndcond}
\varphi(N,x_N)=\frac{\nu_N(x_N)}{p_N^*(x_N)}
\end{equation}
which is admissible as a boundary condition for (\ref{eq:generalizedbridgeA}) since it is nonnegative and we are only considering distributions in ${\mathcal P}(\nu_0)$ which satisfy (\ref{support}). With this choice of $\varphi(N,x_N)$, $\pi^*_{ij}(t)$ minimize $J''(P)$ over ${\mathcal P}_S(\nu_0)$.

Define
\begin{equation}\label{defhat}\hat{\varphi}(t,i):=\frac{p^*(t,i)}{\varphi(t,i)}.
\end{equation}
Using (\ref{FPopt}), (\ref{opttransition1}) and (\ref{defhat}), we get
\begin{eqnarray}\nonumber
\hat{\varphi}(t+1,j)&=&\frac{p^*(t+1,j)}{\varphi(t+1,j)}=\frac{\sum_im_{ij}(t)\frac{\varphi(t+1,j)}{\varphi(t,i)}p^*(t,i)}{\varphi(t+1,j)}\\&=&\sum_im_{ij}(t)\frac{p^*(t,i)}{\varphi(t,i)}=\sum_i\pi^*_{ij}(t)\hat{\varphi}(t,i),
\end{eqnarray}
namely $\hat{\varphi}(t,i)$ is {\em space-time co-harmonic} satisfying (\ref{eq:generalizedbridgeB}). While (\ref{defhat}) alone implies (\ref{eq:generalizedbridgeC}), (\ref{bndcond}) and (\ref{defhat}) imply that (\ref{eq:generalizedbridgeD}) is verified.
\end{proof}

In view of (\ref{defhat}), at each time $t=0,1,\ldots,N$  the marginal $p_t^*$ of the solution factorizes as
\begin{equation}\label{fact}
p^*(t,i)=\varphi(t,i)\hat{\varphi}(t,i).
\end{equation}

The final condition (\ref{eq:generalizedbridgeD}) for the Schr\"{o}dinger system  is different from the standard one, see e.g. \cite{CGPT2}. We get from (\ref{eq:generalizedbridgeD})
\begin{equation}\label{bndconditions2' }\varphi(N,x_N)=\sqrt{\frac{\nu_N(x_N)}{\hat{\varphi}(N,x_N)}}.
\end{equation}
Let $\varphi(t)$ and $\hat{\varphi}(t)$ denote the column vectors with entries $\varphi(t,i)$ and $\hat{\varphi}(t,i)$, respectively, with $i\in\mathcal X$. In matrix form, (\ref{eq:generalizedbridgeA}), (\ref{eq:generalizedbridgeB}) and (\ref{opttransition1}) read
\begin{subequations}\label{eq:notations}
\begin{equation}
\varphi(t)=M(t)\varphi(t+1),\; ~~\hat{\varphi}(t+1)=M(t)^T\hat{\varphi}(t),
\end{equation}
and
\begin{equation}\label{matrixtransition}
	\Pi(t)=[\pi_{ij}(t)]=\diag(\varphi(t))^{-1}M(t)\diag(\varphi(t+1)).
\end{equation}
\end{subequations}
Notice that the condition (\ref{bndcond}) involves $p_N^*$ which is defined through (\ref{FPopt}). The latter, in turn, depends on the transition probabilities (\ref{opttransition1}) which require the knowledge of $\varphi(t)$ for $t=0,1,\ldots,N-1$. Thus, it is not clear if the  whole procedure is well-posed.   This will be established in Section \ref{genSch} by proving that indeed  system (\ref{eq:generalizedbridge}) has a (unique) solution.

\section{Background: Hilbert's projective metric}\label{Hilbert}
This metric dates back to 1895 \cite{hilbert1895gerade}. A crucial contractivity result that permits to establish existence of solutions  of equations on cones (such as the Perron-Frobenius theorem)  was proven by Garrett Birkhoff  in 1957 \cite{birkhoff1957extensions}. Important extensions of Birkhoff's result to nonlinear maps were provided by Bushell \cite{bushell1973projective,bushell1973hilbert}. Various other applications of the Birkhoff-Bushell result have been developed such as to positive integral operators and to positive definite matrices \cite{bushell1973hilbert, lemmens2013birkhoff}. More recently, this geometry has proven useful in various problems concerning  communication and computations over networks (see \cite{tsitsiklis1986distributed} and the work of Sepulchre and collaborators \cite{sepulchre2010consensus,sepulchre2011contraction,BFS} on consensus in non-commutative spaces and metrics for spectral densities) and in statistical quantum theory \cite{reeb2011hilbert}. A recent survey on the applications in analysis is \cite{lemmens2013birkhoff}. The use of the Hilbert metric is crucial in the nonlinear Frobenius-Perron theory \cite{lemmens2012nonlinear}.
A considerable further extension of the Perron-Frobenius theory beyond linear positive systems and monotone systems has been recently proposed in \cite{FS}.

Taking advantage of the Birkhoff-Bushell results on contractivity of linear and nonliner maps on cones,  we showed in \cite{GP} that the Schr\"odinger bridge for Markov chains and quantum channels can be efficiently obtained from the fixed-point of a map which is contractive in the {\em Hilbert metric}. This result extended \cite{FL} which deals with scaling of nonnegative matrices. In \cite{CGP}, it was shown that a similar approach can be taken in the context of diffusion processes leading to i)  a new proof of a classical result on SBP and ii) providing an efficient computational scheme for both, SBP and OMT. This new computational approach can be effectively employed, for instance, in image interpolation.

Following \cite{bushell1973hilbert}, we recall some basic concepts and results of this theory.

Let $\cS$ be a real Banach space and let $\cK$ be a closed solid cone in $\cS$, i.e., $\cK$ is closed with nonempty interior ${\rm int}\cK$and is such that $\cK+\cK\subseteq \cK$, $\cK\cap -\cK=\{0\}$ as well as $\lambda \cK\subseteq \cK$ for all $\lambda\geq 0$. Define the partial order
\[
x\preceq y \Leftrightarrow y-x\in\cK,\quad x< y \Leftrightarrow y-x\in{\rm int}\cK
\]
and for $x,y\in\cK_0:=\cK\backslash \{0\}$, define
\begin{eqnarray*}
M(x,y)&:=&\inf\, \{\lambda\,\mid x\preceq \lambda y\}\\
m(x,y)&:=&\sup \{\lambda \mid \lambda y\preceq x \}.
\end{eqnarray*}
Then, the Hilbert metric is defined on $\cK_0$ by
\[
d_H(x,y):=\log\left(\frac{M(x,y)}{m(x,y)}\right).
\]
Strictly speaking, it is a {\em projective} metric since it is invariant under scaling by positive constants, i.e.,
$d_H(x,y)=d_H(\lambda x,\mu y)=d_H(x,y)$ for any $\lambda>0, \mu>0$ and $x,y\in{\rm int}\cK$. Thus, it is actually a  distance between rays. If $U$ denotes the unit sphere in $\cS$, $\left({\rm int}\cK\cap U,d_H\right)$ is a metric space.
\begin{ex}\label{ex1} Let $\cK=\R^n_+=\{x\in\R^n: x_i\ge 0\}$ be the positive orthant of $\R^n$. Then, for $x,y\in{\rm int}\R^n_+$, namely with all positive components,
\[M(x,y)=\max_i\{x_i/y_i\}, \quad m(x,y)=\min_i\{x_i/y_i\},
\]
and
\[d_H(x,y)=\log\max\{x_iy_j/y_ix_j\}.
\]
\end{ex}
Another very important example for applications in many diverse areas of statistics, information theory, control,etc. is the cone of Hermitian, positive semidefinite matrices.
\begin{ex}\label{ex2}
Let $\cS=\{X=X^\dagger\in\C^{n\times n}\}$, where $\dagger$ denotes here transposition plus conjugation and, more generally, adjoint. Let $\cK=\{X\in\cS: X\ge 0\}$ be the positive semidefinite matrices. Then, for $X,Y\in{\rm int}\cK$, namely positive definite, we have
\[d_H(X,Y)=\log\frac{\lambda_{\max}\left(XY^{-1}\right)}{\lambda_{\min}\left(XY^{-1}\right)}=\log\frac{\lambda_{\max}\left(Y^{-1/2}XY^{-1/2}\right)}{\lambda_{\min}\left(Y^{-1/2}XY^{-1/2}\right)}.
\]
It is closely connected to the Riemannian (Fisher-information) metric
\begin{eqnarray}\nonumber d_R(X,Y)&=&\|\log\left(Y^{-1/2}XY^{-1/2}\right)\|_{F}\\&=&
\sqrt{\sum_{i=1}^n[\log\lambda_i\left(Y^{-1/2}XY^{-1/2}\right)]^2}.\nonumber
\end{eqnarray}
\end{ex}

A map $\cE:\cK\rightarrow\cK$ is called {\em non-negative}. It is called {\em positive} if $\cE:{\rm int}\cK\rightarrow{\rm int}\cK$. If $\cE$ is positive and $\cE(\lambda x)=\lambda^p\cE(x)$ for all $x\in{\rm int}\cK$ and positive $\lambda$, $\cE$ is called {\em positively homogeneous of degree $p$} in ${\rm int}\cK$.
For a positive map $\cE$, the {\em projective diameter} is befined by
\begin{eqnarray*}
\Delta(\cE):=\sup\{d_H(\cE(x),\cE(y))\mid x,y\in {\rm int}\cK\}
\end{eqnarray*}
and the {\em contraction ratio} by
\begin{eqnarray*}
k(\cE):=\inf\{\lambda: \mid d_H(\cE(x),\cE(y))\leq \lambda d_H(x,y),\forall x,y\in{\rm int}\cK\}.
\end{eqnarray*}
Finally, a map $\cE:{\cS}\rightarrow\cS$ is called {\em monotone increasing} if $x\le y$ implies $\cE(x)\le\cE(y)$.
\begin{thm}[\cite{bushell1973hilbert}] \label{poshom}Let $\cE$ be a monotone increasing positive mapping which is positive
homogeneous of degree $p$ in ${\rm int}\cK$. Then the contraction $k(\cE)$ does not exceed $p$. In particular, if $\cE$ is a positive linear mapping, $k(\cE)\le1$.
\end{thm}

\begin{thm}[\cite{{birkhoff1957extensions},bushell1973hilbert}]\label{BBcontraction}
Let $\cE$ be a positive linear map. Then
\begin{equation}\label{condiam}
k(\cE)=\tanh(\frac{1}{4}\Delta(\cE)).
\end{equation}
\end{thm}

\begin{thm}[\cite{bushell1973hilbert}]\label{BBcontraction2}
Let $\cE$ be either
\begin{description}
\item  [a.] a monotone increasing positive mapping which is positive homogeneous of degree $p (0<p<1)$ in ${\rm int}\cK$, or
\item [b.] a positive linear mapping with finite projective diameter.
\end{description}
Suppose the metric space $Y=\left({\rm int}\cK\cap U,d_H\right)$ is complete. Then, in case $(a)$ there exists a unique $x\in{\rm int}\cK$ such that $\cE(x)=x$, in case $(b)$ there exists a unique positive eigenvector of $\cE$ in $Y$.
\end{thm}
This result provides a far-reaching generalization of the celebrated Perron-Frobenius theorem \cite{birkhoff1962Perron-Frobenius}. Notice that in both Examples \ref{ex1} and \ref{ex2}, the space $Y=\left({\rm int}\cK\cap U,d_H\right)$ is indeed complete \cite{bushell1973hilbert}.

There are  other metrics which are contracted by positive monotone maps. For instance, the closely related {\em Thompson metric} \cite{Tho}
\[d_T(x,y)=\log\max\{M(x,y),m^{-1}(x,y)\}.
\]
The Thompson metric is a {\em bona fide} metric on $\cK$. It has been, for instance, employed in \cite{LW,CPSV,BFS}.

\section{Solution to the generalized Schr\"{o}dinger system}\label{genSch}
Let $G=M(N-1) M(N-2)\cdots M(1) M(0)$ and assume that all its elements  $g_{ij}$ are positive. Let us introduce the following maps on $\R^n_+$:
\begin{eqnarray*}
\cE\;:\; &x\mapsto y&=\sum_jg_{ij}x_j,\\
\cE^\dagger \;:\; &x\mapsto y&=\sum_ig_{ij}x_i,\\
\cD_0\;:\;&x\mapsto y&= \frac{\nu_0}{x}\\
\cD_N\;:\;&x\mapsto y&= \frac{\nu_N}{x}
\end{eqnarray*}
where division of vectors is componentwise\footnote{Our use of the adjoint for the map $\cE$ is consistent with the standard notation in diffusion processes where the Fokker-Planck (forward) equation involves the adjoint of the {\em generator} appearing in the backward Kolmogorov equation.}.

\begin{lemma}\label{linear}
Consider the maps $\cE$ and $\cE^\dagger$. We have the following bounds on their contraction ratios:
\begin{equation}\label{strictcon}
k(\cE)=k(\cE^\dagger)=\tanh(\frac{1}{4}\Delta(\cE))<1.
\end{equation}
\end{lemma}
\begin{proof}
Observe that $\cE$ is a positive {\em linear} map and its projective diameter is
\begin{eqnarray*}
\Delta(\cE)&=&\sup\{d_H(\cE(x),\cE(y)) \mid x_i>0,\,y_i>0\}\\
&=&\sup\{\log\left(\frac{g_{ij}g_{k\ell}}{g_{i\ell}g_{kj}}\right)\mid1\leq  i,j,k,\ell\leq n\}.
\end{eqnarray*}
It is finite since all entries $g_{ij}$'s are positive. It now follows from Theorem \ref{BBcontraction} that its contraction ratio satisfies (\ref{strictcon}). Similarly for the adjoint map $\cE^\dagger$. \end{proof}
\begin{lemma}\label{inversion}
\[k(\cD_0)\le1,\quad k(\cD_N)\le1\]
\end{lemma}
\begin{proof}See \cite[p.033301-10]{GP}.\end{proof}
\begin{lemma}\label{sqrt} Let $\mathcal R:\R^n_+\rightarrow\R^n_+$ be the map which associates to the vector $x$ with components $x_i$ to the vector with components $\sqrt{x_i}$. Then
\begin{equation}\label{CRSR}
k(\mathcal R)=1/2.
\end{equation}
\end{lemma}
\begin{proof} Let $x,y\in{\rm int}\R^n_+$. In view of Example\ref{ex1} and using the properties of the square root,
\[d_H(\mathcal R(x),\mathcal R(y))\hspace*{-2pt}=\hspace*{-2pt}\log\max\{\sqrt{(x_iy_j/y_ix_j)}\}\hspace*{-2pt}=\hspace*{-2pt}(1/2)d_H(x,y).
\]
\end{proof}

\begin{thm}\label{compothm} The composition
\begin{equation}\label{composition}\cC:=\cE^\dagger\circ\cD_0\circ \cE\circ \mathcal R\circ\cD_N
\end{equation}
contracts the Hilbert metric with contraction ratio $k(\cC)<(1/2)$, namely
\[
d_H(\cC(x),\cC(y))<(1/2)d_H(x,y), \quad \forall x,y\in{\rm int}\R^n_+.
\]
\end{thm}
\begin{proof}The result follows at once from Lemmas \ref{linear}, \ref{inversion}, \ref{sqrt}.
\end{proof}
 We have the following result.
\begin{thm} \label{fundtheorem} Assume that the matrix $G=M(N-1) M(N-2)\cdots M(1) M(0)$ all positive elements  $g_{ij}$. Let $\nu_0$ and $\nu_N$ be any two probability distributions on $\mathcal X$.  Then, there exist a unique choice of the vectors
$\varphi(0),\,\hat\varphi(N)$ with positive entries such that
\begin{subequations}\label{eq:iteration}
\begin{align}\label{eq:iterationa}
&\varphi(t,i)=\sum_{j}m_{ij}(t)\varphi(t+1,j), \; 0\le t\le N-1\\
&\hat\varphi(t+1,j)=\sum_{i}m_{ij}(t)\hat\varphi(t,i), \; 0\le t\le N-1\label{eq:iterationb}\\
\label{eq:iterationc}
&\varphi(0,x_0)\hat\varphi(0,x_0)=\nu_0(x_0),\\\label{eq:iterationd}
&\varphi(N,x_N)^2\hat\varphi(N,x_N)=\nu_N(x_N).
\end{align}
\end{subequations}
\end{thm}
\begin{proof} Consider the iteration
\begin{equation}\label{iteration}
(\hat\varphi(N,\cdot))_{\rm next}=\cC(\hat\varphi(N,\cdot))
\end{equation}
Notice that the componentwise divisions of $\cD_0$ and $\cD_N$ are well defined. Indeed, even when $\hat\varphi(0)$ ($\varphi(N)$) has zero entries, $\hat\varphi(N)$ ($\varphi(0)$) has all positive entries since the elements of $G$ are all positive.
Since $\cC$ is strictly contractive in the Hilbert metric, the iteration \eqref{iteration} would converge to a ray that is invariant under $\cC$. Let $\phi$ be any positive vector on this ray, then
\begin{equation*}
\cC(\phi)= \lambda \phi
\end{equation*}
for some positive number $\lambda$. Now let $\hat\varphi(N) =\lambda^2\phi$. Then it is straightforward to verify
\begin{equation}
	\cC(\hat\varphi(N)) = \lambda\cC(\phi) = \lambda^2\phi =\hat\varphi(N).
\end{equation}
Moreover, this is the unique vector that satisfies the above condition. Define
	\begin{eqnarray*}
		\varphi(N) &=& \sqrt{\frac{\nu_N}{\hat\varphi(N)}}
		\\
		\varphi(0) &=& \cE(\varphi(N))
		\\
		\hat\varphi(0) &=& \frac{\nu_0}{\varphi(0)}
	\end{eqnarray*}
and $\varphi(t), \hat\varphi(t)$ according to \eqref{eq:iterationa}-\eqref{eq:iterationb}, then clearly these vectors are consistent with the Schr\"{o}dinger system (\ref{eq:iterationa})-(\ref{eq:iterationb})-(\ref{eq:iterationc})-(\ref{eq:iterationd}).
\end{proof}

\section{An algorithm contracting Hilbert's metric and some extensions}\label{algo}
Let $\mathbf{1}^\dagger=(1,1,\ldots,1)$.
The iteration in (\ref{iteration}) suggests the following algorithm:
\begin{description}
\item [a.] Set $x=x(0)=\mathbf{1}$;
\item [b.] Set $x_{\rm next}=\cC(x)$;
\item [c.] Iterate until you reach a fixed point $\bar{x}=\cC(\bar{x})$ ({\em stopping criterion: $|\bar x-\cC\bar x|<10^{-4}$});
\item [d.] Set $\hat\varphi(N)=\bar{x}$;
\item [e.] Use
\begin{equation}\label{eq:stepe}
\varphi(N,x_N)=\sqrt{\frac{\nu_N(x_N)}{\hat{\varphi}(N,x_N)}}
\end{equation}
to compute $\varphi(N)$ and then (\ref{eq:iterationa}) to compute $\varphi(t)$ for $t=N-1,N-2,\ldots, 1,0$;
\item [f.] Compute the optimal transition probabilities $\pi^*_{ij}(t)$ according to (\ref{opttransition1});
\item [g.] The solution to Problem \ref{relbridge} is the time inhomogeneous Markovian distribution (\ref{optmeasure}) with initial marginal $\nu_0$ and transition probabilities $\pi^*_{ij}(t)$.
\end{description}

The assumption that the elements  $g_{ij}$ of the matrix $G=M(N-1) M(N-2)\cdots M(1) M(0)$ be all positive can be relaxed. For instance, if both $\nu_0$ and $\nu_N$ are everywhere positive on $\mathcal X$, it suffices that $M$ has at least one positive element in each row and column to guarantee that the componentwise divisions of $\cD_0$ and $\cD_N$ are well defined. In that case, Theorems \ref{mainthm} and \ref{compothm} hold true and the algorithm of this section applies.

Our analysis and algorithm can be generalized to the cost function 
	\begin{equation}\label{eq:weightedD}
		\D(P\|\fM)+\eta\D(p_N\|\nu_N)
	\end{equation}
for any $\eta\ge 0$. In this case, we only need change \eqref{eq:generalizedbridgeD} in the Schr\"odinger system to
	\[
		\varphi(N,j)^{\frac{\eta+1}{\eta}}\hat\varphi(N,j)=\nu_N(j)
	\]
and \eqref{eq:stepe} to
	\[
		\varphi(N,x_N)=\left(\frac{\nu_N(x_N)}{\hat{\varphi}(N,x_N)}\right)^{\frac{\eta}{\eta+1}}
	\]
in the algorithm. The convergence rate is strictly upper bounded by $\frac{\eta}{\eta+1}$. The parameter $\eta$ measures the significance of the penalty term $\D(p_N\|\nu_N)$. When $\eta$ goes to infinity, we recover the traditional Schr\"odinger bridge. The upper bound is $1$ in this case. On the other hand, when $\eta=0$, the solution is trivial in view of (\ref{relentrdec2}). It is the Markov process  with kernel $M(t)$ (assuming that all $M(t)$ are stochastic matrices) and initial distribution $\nu_0$. Indeed, $\frac{\eta}{\eta+1}=0$ implies $\varphi(N,\cdot)=1$ on $\mathcal X$. In view of (\ref{eq:iterationa}), we get  $\varphi(t,i)\equiv 1$. This is intuitive and we do not need to run the algorithm to solve the problem when $\eta=0$.

\begin{ex}
Consider the graph in Figure \ref{fig:graph}. We seek to transport masses from initial distribution $\nu_1=\delta_1$ to target distribution $\nu_N=1/2\delta_6+1/2\delta_9$. The step $N$ is set to be $3$ or $4$.
\begin{figure}[h]
\centering
\includegraphics[width=0.35\textwidth]{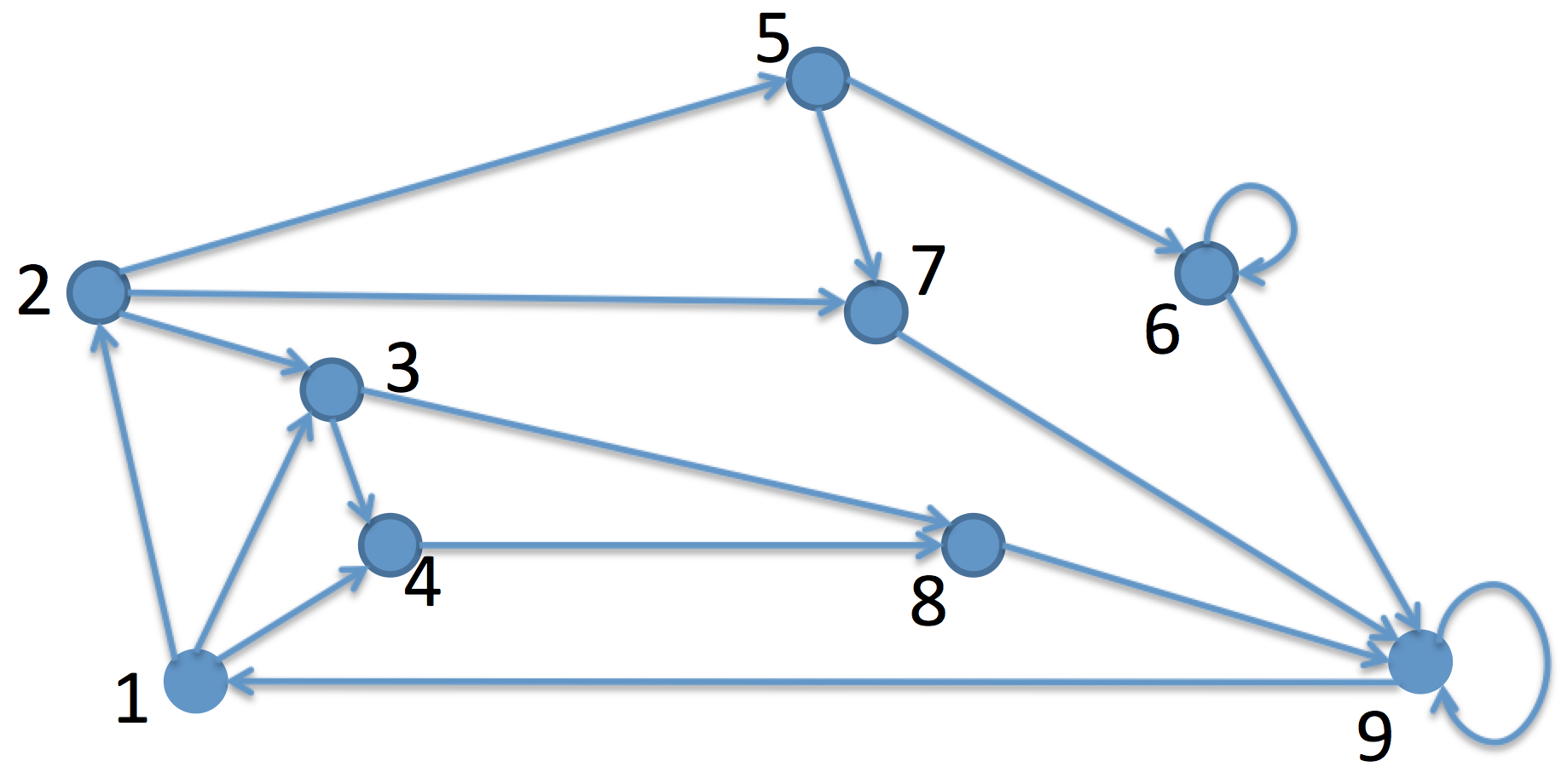}
\caption{transport graph}
\label{fig:graph}
\end{figure}
When $N=3$, the evolution of mass distribution by solving Problem \ref{relbridge} is given by
 \[
      \tiny
       \left[
       \begin{matrix}
       1 & 0 & 0 & 0 & 0 & 0 & 0 & 0 & 0\\
       0 & 0.5865 & 0.2067 & 0.2067 & 0 & 0 & 0 & 0 & 0\\
       0 & 0 & 0 & 0 & 0.3798 & 0 & 0.2067 & 0.4135 & 0\\
       0 & 0 & 0 & 0 & 0 & 0.3798 & 0 & 0 & 0.6202
       \end{matrix}
       \right],
\]
where the four rows of the matrix show the mass distribution at time step $t=0, 1, 2 ,3$ respectively. The prior law $M$ is taken to be the Rulle Bowen random walk \cite{CGPT1}. The mass spreads out before reaching nodes $6$ and $9$. Due to the soft terminal constraint, the terminal distribution is not equal to $\nu_N$. When we allow for more steps $N=4$, the mass spreads even more before reassembling at nodes $6, 9$, as shown below,
 \[
       \tiny
       \left[
       \begin{matrix}
       1 & 0 & 0 & 0 & 0 & 0 & 0 & 0 & 0\\
       0 & 0.6941 & 0.2040 & 0.1020 & 0 & 0 & 0 & 0 & 0\\
       0 & 0 & 0.1020 & 0.1020 & 0.4901 & 0 & 0.1020 & 0.2040 & 0\\
       0 & 0 & 0 & 0 & 0 & 0.3881 & 0.1020 & 0.2040 & 0.3059\\
       0 & 0 & 0 & 0 & 0 & 0.2862 & 0 & 0 & 0.7138
       \end{matrix}
       \right].
\]
The terminal distribution is again not equal to $\nu_N$. However, if we increase the penalty on $\D(p_N\|\nu_N)$, then the difference between $p_N$ and $\nu_N$ becomes smaller, as can be seen below, which is the distribution evolution when $\eta=10$ in \eqref{eq:weightedD}
 \[
       \tiny
       \left[
       \begin{matrix}
       1 & 0 & 0 & 0 & 0 & 0 & 0 & 0 & 0\\
       0 & 0.7679 & 0.1547 & 0.0774 & 0 & 0 & 0 & 0 & 0\\
       0 & 0 & 0.0774 & 0.0774 & 0.6132 & 0 & 0.0774 & 0.1547 & 0\\
       0 & 0 & 0 & 0 & 0 & 0.5359 & 0.0774 & 0.1547 & 0.2321\\
       0 & 0 & 0 & 0 & 0 & 0.4585 & 0 & 0 & 0.5415
       \end{matrix}
       \right].
\]
\end{ex}

\section{Final comments}

Since the work of Mikami, Thieullen, Leonard, Cuturi \cite{Mik,mt,MT,leo2,leo,Cuturi}, a large number of papers have appeared where Schr\"{o}dinger bridge problems are viewed as regularization of the important Optimal Mass Transport (OMT) problem, see e.g., \cite{BCCNP,CGP4,CGP,CGP5,LYO,Alt,CPSV}. This is, of course, interesting and extremely effective as OMT is computationally challenging \cite{AHT,BB}. Nevertheless, one should not forget that Schr\"{o}dinger bridge problems have at least two other important motivations: The first is Schr\"odinger's original ``hot gas experiment'' model, namely {\em large deviations of the empirical distribution on paths} \cite{F2}. The second is a {\em maximum entropy principle in statistical inference}, namely choosing the a posterior distribution so as to make the fewest number of assumptions about what is beyond the available information. This inference method has been noticeably developed over the years by Jaynes, Burg, Dempster and Csisz\'{a}r \cite{Jaynes57,Jaynes82,BURG1,BLW,Dem,csiszar0,csiszar1,csiszar2}. It is this last concept which largely inspired the original approach taken in this paper and in \cite{CGPT1,CGPT2} although connections to OMT were made there. The prior mass distribution on paths may namely simply encode the topological information of the network or that plus the length of each link and is not necessarily a probability distribution. In this paper, in particular, we have considered a relaxed version of the problem where the final distribution only need to be close to a desired one. This has been formalized by adding to the criterion the Kullback-Leibler distance between the final distribution and the desired one. We have shown that the solution can be otained solving a Schr\"{o}dinger system with different terminal condition. An iterative algorithm contracting the Hilbert metric with contraction ratio less than $1/2$ to compute the solution has been provided as well.

\end{document}